\documentclass{amsart}
\usepackage{setspace}
\usepackage{a4}
\usepackage{amsthm}
\usepackage{latexsym}
\usepackage{amsfonts}
\usepackage{graphicx}
\usepackage{textcomp}
\usepackage{cite}
\usepackage{enumerate}
\usepackage{amssymb}
\usepackage{hyperref}
\usepackage{amsmath}
\usepackage{tikz}
\usepackage[mathscr]{euscript}
\usepackage{mathtools}
\newtheorem{theorem}{Theorem}[section]

\newtheorem{definition}[theorem]{Definition}
\newtheorem{example}[theorem]{Example}

\newtheorem{remark}[theorem]{Remark}

\title{This is the title}
\usepackage{fancyhdr}

\begin{document}
\hrule\hrule\hrule\hrule\hrule
\vspace{0.3cm}	
\begin{center}
{\bf{Non-Archimedean Heisenberg-Robertson-Schrodinger and Non-Archimedean Maccone-Pati Uncertainty Principles}}\\
\vspace{0.3cm}
\hrule\hrule\hrule\hrule\hrule
\vspace{0.3cm}
\textbf{K. Mahesh Krishna}\\
School of Mathematics and Natural Sciences\\
Chanakya University Global Campus\\
NH-648, Haraluru Village\\
Devanahalli Taluk, 	Bengaluru  North District\\
Karnataka State, 562 110, India \\
Email: kmaheshak@gmail.com\\

Date: \today
\end{center}

\hrule\hrule
\vspace{0.5cm}
\textbf{Abstract}: Let  $\mathcal{X}$ be a non-Archimedean Hilbert space over a conjugated non-Archimedean valued field $\mathbb{K}$ with $\operatorname{Char}(\mathbb{K})\neq 2$. Let $A:	\mathcal{D}(A)\subseteq \mathcal{X}\to  \mathcal{X}$ and $B:	\mathcal{D}(B)\subseteq \mathcal{X}\to  \mathcal{X}$  be  possibly unbounded self-adjoint linear  operators. For $x \in \mathcal{D}(A)$  with $\langle x, x \rangle =1$, define $	\Delta _x(A)\coloneqq \|Ax- \langle Ax, x \rangle x \|.$
Then for all $x \in \mathcal{D}(AB)\cap  \mathcal{D}(BA)$ with $\langle x,  x \rangle =1$, we show that 
\begin{align*}
(1)	\quad \Delta_x(A)+\Delta_x(B)\geq \max\{\Delta_x(A), \Delta_x(B)\}\geq  \frac{\sqrt{\bigg|\big\langle [A,B]x, x \big\rangle ^2+\big(\langle \{A,B\}x, x \rangle -2\langle Ax, x \rangle\langle Bx, x \rangle\big)^2\bigg|}}{\sqrt{|2|}}
\end{align*}
and 
\begin{align*}
(2)	\quad  \Delta_x(A)+\Delta_x(B)\geq 	\max\{\Delta_x(A), \Delta_x(B)\} \geq      	|\langle (A+B)x, y \rangle |, \quad \forall y \in \mathcal{X} \text{ satisfying } \|y\|\leq 1, \langle x, y \rangle =0.
\end{align*}
Inequality (1) may be regarded as a non-Archimedean analogue of the Heisenberg-Robertson-Schrodinger uncertainty principle, while Inequality (2) serves as a non-Archimedean analogue of the Maccone-Pati uncertainty principle.



\hrule

\hrule
\section{Introduction}

Let $\mathcal{H}$ be a complex Hilbert space and $A$ be a possibly unbounded self-adjoint linear  operator defined on the domain $\mathcal{D}(A)\subseteq \mathcal{H}$. For $h \in \mathcal{D}(A)$ with $\|h\|=1$, define the uncertainty of $A$ at the point $h$ as 
\begin{align*}
	\Delta _h(A)\coloneqq \|Ah-\langle Ah, h \rangle h \|=\sqrt{\|Ah\|^2-\langle Ah, h \rangle^2}. 
\end{align*}
In 1929, Robertson \cite{ROBERTSON} derived the following mathematical form of the uncertainty principle of Heisenberg derived in 1927 \cite{HEISENBERG}. Recall that for two operators $A:	\mathcal{D}(A)\subseteq \mathcal{H}\to  \mathcal{H}$ and $B:	\mathcal{D}(B)\subseteq \mathcal{H}\to  \mathcal{H}$, we define the commutator $[A,B] \coloneqq AB-BA$ and the anti-commutator  $\{A,B\}\coloneqq AB+BA$.
\begin{theorem} \cite{ROBERTSON, HEISENBERG, VONNEUMANNBOOK, DEBNATHMIKUSINSKI} \label{RHT} (\textbf{Heisenberg-Robertson Uncertainty Principle})
Let  $A:	\mathcal{D}(A)\subseteq \mathcal{H}\to  \mathcal{H}$ and $B:	\mathcal{D}(B)\subseteq \mathcal{H}\to  \mathcal{H}$  be self-adjoint possibly unbounded linear operators. Then for all $h \in \mathcal{D}(AB)\cap  \mathcal{D}(BA)$ with $\|h\|=1$, we have 
\begin{align}\label{HR}
 \frac{1}{2} \left(\Delta _h(A)^2+	\Delta _h(B)^2\right) \geq \left(\frac{\Delta _h(A)+	\Delta _h(B)}{2}\right)^2 \geq  \Delta _h(A)	\Delta _h(B)   \geq  \frac{1}{2}|\langle [A,B]h, h \rangle |.
\end{align}
\end{theorem}
In 1930, Schrodinger made the following improvement of Inequality (\ref{HR}).
\begin{theorem} \cite{SCHRODINGER, BERTLMANNFRIIS}
(\textbf{Heisenberg-Robertson-Schrodinger  Uncertainty Principle}) \label{HRST}
Let  $A:	\mathcal{D}(A)\subseteq \mathcal{H}\to  \mathcal{H}$ and $B:	\mathcal{D}(B)\subseteq \mathcal{H}\to  \mathcal{H}$  be self-adjoint possibly unbounded linear operators. Then for all $h \in \mathcal{D}(AB)\cap  \mathcal{D}(BA)$ with $\|h\|=1$, we have 
\begin{align*}
  \Delta _h(A)	\Delta _h(B) &\geq  |\langle Ah, Bh \rangle-\langle Ah, h \rangle \langle Bh, h \rangle|\\
  &= \frac{\sqrt{|\langle [A,B]h, h \rangle |^2+|\langle \{A,B\}h, h \rangle -2\langle Ah, h \rangle\langle Bh, h \rangle|^2}}{2}\\
  &=\frac{\sqrt{(\langle \{A,B\}h, h \rangle -2\langle Ah, h \rangle\langle Bh, h \rangle)^2-\langle [A,B]h, h \rangle ^2}}{2}.
\end{align*}	
\end{theorem}
 Surprisingly, in 2014,  Maccone and Pati derived the following  uncertainty principle which works for any unit  vector which is  orthogonal to given unit vector  \cite{MACCONEPATI}. 
\begin{theorem}  \cite{MACCONEPATI} (\textbf{Maccone-Pati Uncertainty Principle}) \label{MP}
Let  $A:	\mathcal{D}(A)\subseteq \mathcal{H}\to  \mathcal{H}$ and $B:	\mathcal{D}(B)\subseteq \mathcal{H}\to  \mathcal{H}$  be self-adjoint possibly unbounded linear operators. Then for all $h \in \mathcal{D}(A)\cap  \mathcal{D}(B)$ with $\|h\|=1$, we have 
\begin{align*}
	\Delta _h(A)^2+	\Delta _h(B)^2 \geq      \frac{1}{2} \left(|\langle (A+B)h, k \rangle|^2+|\langle (A-B)h, k \rangle|^2\right), \quad \forall k \in \mathcal{H} \text{ satisfying } \|k\|=1, \langle h, k \rangle =0.
\end{align*}		
\end{theorem}
	The theory of non-Archimedean Hilbert spaces provides a natural non-Archimedean analogue of classical Hilbert space theory and has connections with non-Archimedean functional analysis and non-Archimedean mathematical physics. Since uncertainty relations are among the most fundamental inequalities in operator theory on Hilbert spaces, it is natural to ask whether analogous inequalities can be formulated in the non-Archimedean setting and, if so, how the ultrametric structure influences their form. The purpose of this paper is not to derive physical uncertainty relations in a non-Archimedean quantum-mechanical model. Rather, we investigate operator inequalities on non-Archimedean Hilbert spaces that parallel the algebraic structure of the classical Heisenberg-Robertson-Schrodinger and Maccone-Pati inequalities. We show that the non-Archimedean geometry forces these inequalities to take a different form, with maxima replacing several product-type estimates that arise in the Archimedean setting.

  Non-Archimedean functional analysis and non-Archimedean Hilbert spaces have been studied for many decades, beginning with the work of Kalisch. Moreover, non-Archimedean methods have appeared in several mathematical models related to quantum theory and non-Archimedean analysis \cite{VLADIMIROVVOLOVICHZELENOV, AMP1, AMP2}. Nevertheless, uncertainty-type inequalities for operators on non-Archimedean Hilbert spaces do not appear to have been systematically investigated. The purpose of this note is to initiate such a study.

We emphasize that the present work is purely mathematical. Although the terminology is inspired by the uncertainty principles of quantum mechanics, we do not claim a physical interpretation of the derived inequalities. Rather, the goal is to identify non-Archimedean operator inequalities that mirror the algebraic structure of the classical uncertainty relations.

We investigate whether classical uncertainty principles admit meaningful analogues in the setting of non-Archimedean Hilbert spaces. Working over a conjugated non-Archimedean valued field, we define the uncertainty of an operator at a vector and establish inequalities that may be regarded as non-Archimedean counterparts of the Heisenberg–Robertson–Schrodinger and Maccone–Pati uncertainty principles. The proofs rely on the ultrametric inequality and the axioms of a non-Archimedean inner product. While the results are formally parallel to the classical case, their conceptual significance lies in illustrating how incompatibility-type bounds manifest in non-Archimedean analysis. We also provide examples.

\section{Non-Archimedean Heisenberg-Robertson-Schrodinger  Uncertainty Principle and  Non-Archimedean Maccone-Pati uncertainty principle}
We are going to consider the following notion of non-Archimedean Hilbert space, which contains the classical p-adic Hilbert space as a special case \cite{KALISCH}.
\begin{definition}
	Let $\mathbb{K}$ be a non-Archimedean  valued field (with valuation $|\cdot|$). A bijective map $\sigma: \mathbb{K} \to \mathbb{K}$ is said to be a conjugation if following conditions hold.
	\begin{enumerate}[\upshape(i)]
		\item $\sigma(\alpha+\beta)=\sigma(\alpha)+\sigma(\beta)$ for all $\alpha, \beta \in \mathbb{K}$.
		\item  $\sigma(\alpha\beta)=\sigma(\alpha)\sigma(\beta)$ for all $\alpha, \beta \in \mathbb{K}$.
		\item $\sigma(\sigma(\alpha))=\alpha$ for all $\alpha  \in \mathbb{K}$.
		\item $|\sigma(\alpha)|=|\alpha|$ for all $\alpha  \in \mathbb{K}$.
	\end{enumerate}
	In this case, we say $\mathbb{K}$ is  $\sigma$-conjugated. 
\end{definition}
\begin{definition}\label{PADICDEF}
	Let $\mathbb{K}$ be a $\sigma$-conjugated non-Archimedean  valued field (with valuation $|\cdot|$) and $\mathcal{X}$ be a non-Archimedean Banach space (with norm $\|\cdot\|$) over $\mathbb{K}$. We say that $\mathcal{X}$ is a \textbf{non-Archimedean Hilbert space} if there is a map (called as non-Archimedean inner product) $\langle \cdot, \cdot \rangle: \mathcal{X} \times \mathcal{X} \to \mathbb{K}$ satisfying following.
	\begin{enumerate}[\upshape (i)]
		\item If $x \in \mathcal{X}$ is such that $\langle x,y \rangle =0$ for all $y \in \mathcal{X}$, then $x=0$.
		\item $\langle x, y \rangle =\sigma(\langle y, x \rangle)$ for all $x,y \in \mathcal{X}$.
		\item $\langle x+y, z \rangle = \langle x,  z \rangle+\langle y,z\rangle$  for all $x,y,z \in \mathcal{X}$.
		\item $\langle \alpha x, y \rangle =\alpha \langle x,  y \rangle$  for all $x,y \in \mathcal{X}$, for all $\alpha \in \mathbb{K}$.
		\item $|\langle x, y \rangle |\leq \|x\|\|y\|$ for all $x,y \in \mathcal{X}$.
	\end{enumerate}
\end{definition}
Following are  standard examples.
\begin{example}
	Let $d\in \mathbb{N}$ and 	$\mathbb{K}$ be a non-Archimedean  valued field. 	Then $\mathbb{K}^d$ is a non-Archimedean Hilbert space w.r.t. the ultranorm 
	\begin{align*}
		\|(x_j)_{j=1}^d\|\coloneqq \max_{1\leq j \leq d}|x_j|, \quad \forall (x_j)_{j=1}^d\in 	\mathbb{K}^d
	\end{align*}
	and the non-Archimedean inner product 
	\begin{align*}
		\langle (x_j)_{j=1}^d, (y_j)_{j=1}^d\rangle \coloneqq \sum_{j=1}^dx_jy_j, \quad \forall (x_j)_{j=1}^d, (y_j)_{j=1}^d\in 	\mathbb{K}^d.
	\end{align*}
\end{example}
\begin{example}
	Let $\mathbb{K}$ be a non-Archimedean  valued field. Define
	\begin{align*}
		c_0(\mathbb{N}, \mathbb{K})\coloneqq \{(x_n)_{n=1}^\infty:x_n \in \mathbb{K}, \forall n \in \mathbb{N}, \lim_{n\to \infty}x_n=0\}.
	\end{align*} 
	Then $c_0(\mathbb{N}, \mathbb{K})$ is a non-Archimedean Hilbert space w.r.t. the ultranorm 
	\begin{align*}
		\|(x_n)_{n=1}^\infty\|\coloneqq \sup_{n\in \mathbb{N}}|x_n|, \quad \forall (x_n)_{n=1}^\infty\in 	c_0(\mathbb{N}, \mathbb{K})
	\end{align*}
	and the non-Archimedean inner product 
	\begin{align*}
		\langle (x_n)_{n=1}^\infty, (y_n)_{n=1}^\infty\rangle \coloneqq \sum_{n=1}^\infty x_ny_n, \quad \forall (x_n)_{n=1}^\infty, (y_n)_{n=1}^\infty\in 	c_0(\mathbb{N}, \mathbb{K}).
	\end{align*}	
\end{example}
Let $\mathcal{X}, \mathcal{Y}$ be  non-Archimedean Hilbert spaces and $T:\mathcal{X}\to \mathcal{Y}$ be a  linear operator. We say that $T$ is adjointable if there is a  linear operator, denoted by $T^*:\mathcal{Y}\to \mathcal{X}$ such that $\langle Tx,y\rangle =\langle x,T^*y\rangle$, $\forall x \in \mathcal{X}, \forall y \in \mathcal{Y}$. Note that (i) in Definition \ref{PADICDEF} says that adjoint, if exists,  is unique.  An adjointable  linear operator $T:\mathcal{X}\to  \mathcal{X}$ is said to be  self-adjoint  if $T^*=T$. \\
 Let $A$ be a possibly unbounded  linear operator (need not be self-adjoint) defined on domain $\mathcal{D}(A)\subseteq \mathcal{X}$. For $x \in \mathcal{D}(A)$  with $\langle x, x \rangle =1$, define the uncertainty of $A$ at the point $x$ as 
\begin{align*}
	\Delta _x(A)\coloneqq \|Ax- \langle Ax, x \rangle x \|. 
\end{align*}
We now have the non-Archimedean version of Theorem \ref{HRST}.
\begin{theorem}\label{PHRS}
	(\textbf{Non-Archimedean Heisenberg-Robertson-Schrodinger  Uncertainty Principles})
	Let  $\mathcal{X}$ be a non-Archimedean Hilbert space. Let $A:	\mathcal{D}(A)\subseteq \mathcal{X}\to  \mathcal{X}$ and $B:	\mathcal{D}(B)\subseteq \mathcal{X}\to  \mathcal{X}$  be possibly unbounded linear  operators. Then for all $x \in \mathcal{D}(AB)\cap  \mathcal{D}(BA)$ with $\langle x,  x \rangle =1$, we have 
	\begin{enumerate}[\upshape(i)]
		\item 	
		\begin{align*}
			\Delta_x(A) \Delta_x(B)	\geq |\langle Ax, Bx \rangle-\langle Ax, x \rangle \langle x, Bx \rangle|=|\langle Bx, Ax \rangle-\langle x, Ax \rangle \langle Bx,x \rangle|.
		\end{align*}
	In particular, if $A$ and $B$ are self-adjoint, then	
	\begin{align*}
		\Delta_x(A) \Delta_x(B)	\geq |\langle BAx, x \rangle-\langle Ax, x \rangle \langle Bx, x \rangle|=|\langle ABx, x \rangle-\langle Ax, x \rangle \langle Bx, x \rangle|.
	\end{align*}
	\item If $A$ and $B$ are self-adjoint and $\operatorname{Char}(\mathbb{K})\neq 2$, then	
		\begin{align*}
	\Delta_x(A)+\Delta_x(B)\geq 	\max\{\Delta_x(A), \Delta_x(B)\}\geq  \frac{\sqrt{\bigg|\big\langle [A,B]x, x \big\rangle ^2+\big(\langle \{A,B\}x, x \rangle -2\langle Ax, x \rangle\langle Bx, x \rangle\big)^2\bigg|}}{\sqrt{|2|}}.
	\end{align*}
		\item 	If $A$ and $B$ are adjointable and $\operatorname{Char}(\mathbb{K})\neq 2$, then 
		\begin{align*}
		\max\{\Delta_x(A), \Delta_x(B)\}\geq \sqrt{\left |\langle(A^*A+B^*B)x, x \rangle-\frac{\langle (A+B)x, x \rangle\langle x, (A+B)x \rangle+\langle (A-B)x, x \rangle\langle x, (A-B)x \rangle}{2}\right|}.	
		\end{align*}
		In particular, if $A$ and $B$ are self-adjoint, then
		\begin{align*}
			\max\{\Delta_x(A), \Delta_x(B)\}&\geq \sqrt{\left |\langle(A^2+B^2)x, x \rangle-\frac{\langle (A+B)x, x \rangle^2+\langle (A-B)x, x \rangle^2}{2}\right|}\\
			&= \sqrt{\left |\frac{\langle (A+B)^2x, x \rangle+\langle (A-B)^2x, x \rangle-\langle (A+B)x, x \rangle^2-\langle (A-B)x, x \rangle^2}{2}\right|}.
		\end{align*}
		\item 		If $A$ and $B$ are adjointable, then 
		\begin{align*}
		\max\{\Delta_x(A), \Delta_x(B)\}\geq  \sqrt{|\langle (A^*A-B^*B)x, x \rangle-(\langle Ax, x \rangle\langle x,A x \rangle-\langle Bx, x \rangle\langle x, Bx \rangle)|}.	
		\end{align*}
	In particular, if $A$ and $B$ are self-adjoint, then	
	\begin{align*}
		\max\{\Delta_x(A), \Delta_x(B)\}\geq 	 \sqrt{|\langle(A^2-B^2)x, x \rangle-\langle (A+B)x, x \rangle\langle (A-B)x, x \rangle|}.	
	\end{align*}
	\item 	
	\begin{align*}
	\Delta_x(A)+\Delta_x(B)\geq 	\max\{\Delta_x(A), \Delta_x(B)\}\geq \sqrt{|\langle (A+B)x, (A+B)x \rangle-\langle (A+B)x,x \rangle\langle x, (A+B)x \rangle|}.	 
	\end{align*}
\item 
\begin{align*}
\Delta_x(A)+\Delta_x(B)\geq 	\max\{\Delta_x(A), \Delta_x(B)\}\geq\sqrt{|\langle (A-B)x, (A-B)x \rangle-\langle (A-B)x,x \rangle\langle x, (A-B)x \rangle|}.
\end{align*}
	\end{enumerate}
\end{theorem}
\begin{proof}
	Let $x \in \mathcal{D}(AB)\cap  \mathcal{D}(BA)$ be such that $\langle x,  x \rangle =1$. 
	\begin{enumerate}[\upshape(i)]
		\item  	By using the definition of the non-Archimedean inner product,
		\begin{align*}
			\Delta_x(A) \Delta_x(B)	&=\|Ax- \langle Ax, x \rangle x \|\|Bx- \langle Bx, x \rangle x \|\\
			&\geq 	|\langle Ax- \langle Ax, x \rangle x, Bx- \langle Bx, x \rangle x \rangle |\\
			&=|\langle Ax, Bx \rangle-\langle Ax, x \rangle \langle x, Bx \rangle|.
		\end{align*}
		\item By making a direct expansion and simplification, we see that 
		\begin{align*}
			&\big\langle [A,B]x, x \big\rangle ^2+\big(\langle \{A,B\}x, x \rangle -2\langle Ax, x \rangle\langle Bx, x \rangle\big)^2\\
			&=\big(\langle ABx, x \rangle-\langle BAx, x \rangle\big)^2+\big\langle \{A,B\}x, x \big\rangle^2 +4\langle Ax, x \rangle^2 \langle Bx, x \rangle^2-4\langle \{A,B\}x, x \rangle\langle Ax, x \rangle \langle Bx, x \rangle\\
			&=\big(\langle ABx, x \rangle-\langle BAx, x \rangle\big)^2+\big(\langle ABx, x \rangle+\langle BAx, x \rangle\big)^2+4\langle Ax, x \rangle^2 \langle Bx, x \rangle^2\\
			&\quad -4\langle ABx, x \rangle\langle Ax, x \rangle \langle Bx, x \rangle-4\langle BAx, x \rangle\langle Ax, x \rangle \langle Bx, x \rangle\\
			&=2\langle ABx, x \rangle^2+2\langle BAx, x \rangle^2+4\langle Ax, x \rangle^2 \langle Bx, x \rangle^2 -4\langle ABx, x \rangle\langle Ax, x \rangle \langle Bx, x \rangle-4\langle BAx, x \rangle\langle Ax, x \rangle \langle Bx, x \rangle\\
			&=2(\langle ABx, x \rangle-\langle Ax, x \rangle \langle Bx, x \rangle)^2+2(\langle BAx, x \rangle-\langle Ax, x \rangle \langle Bx, x \rangle)^2.
		\end{align*}
		Therefore (by using previous equation and self-adjointness of $A$ and $B$)
		\begin{align*}
			&\left|\big\langle [A,B]x, x \big\rangle ^2+\big(\langle \{A,B\}x, x \rangle -2\langle Ax, x \rangle\langle Bx, x \rangle\big)^2\right|\\
			&=|2|\left|(\langle ABx, x \rangle-\langle Ax, x \rangle \langle Bx, x \rangle)^2+(\langle BAx, x \rangle-\langle Ax, x \rangle \langle Bx, x \rangle)^2\right|\\
			&\leq |2|\max \left\{\left|(\langle ABx, x \rangle-\langle Ax, x \rangle \langle Bx, x \rangle)^2\right|,\left|(\langle BAx, x \rangle-\langle Ax, x \rangle \langle Bx, x \rangle)^2\right|\right\}\\
			&=|2|\max \left\{\left|\langle ABx, x \rangle-\langle Ax, x \rangle \langle Bx, x \rangle\right|^2,\left|\langle BAx, x \rangle-\langle Ax, x \rangle \langle Bx, x \rangle\right|^2\right\}\\
			&\leq |2|\max \left\{ 	\max\{\Delta_x(A)^2, \Delta_x(B)^2\}, 	\max\{\Delta_x(B)^2, \Delta_x(A)^2\}\right\}=|2|\max\{\Delta_x(A)^2, \Delta_x(B)^2\}.
		\end{align*}
		\item By using the non-Archimedean triangle inequality and the definition of the non-Archimedean inner product, 
		\begin{align*}
			\max\{\Delta_x(A), \Delta_x(B)\} &=\max\{\|Ax- \langle Ax, x \rangle x \|, \|Bx- \langle Bx, x \rangle x \|\} \\
			&\geq \max\{\sqrt{|\langle Ax- \langle Ax, x \rangle x, Ax- \langle Ax, x \rangle x\rangle|}, \sqrt{|\langle Bx- \langle Bx, x \rangle x,  Bx- \langle Bx, x \rangle x  \rangle|}\} \\
			&=\max\{\sqrt{|\langle Ax, Ax \rangle-\langle Ax, x \rangle\langle x, Ax \rangle|}, \sqrt{|\langle Bx, Bx \rangle-\langle Bx, x \rangle\langle x, Bx \rangle|}\}\\
			&= \sqrt{\max\{|\langle Ax, Ax \rangle-\langle Ax, x \rangle\langle x, Ax \rangle|, |\langle Bx, Bx \rangle-\langle Bx, x \rangle\langle x, Bx \rangle|\}}\\
			&\geq \sqrt{|\langle Ax, Ax \rangle-\langle Ax, x \rangle\langle x, Ax \rangle+\langle Bx, Bx \rangle-\langle Bx, x \rangle\langle x, Bx \rangle|}\\
			&= \sqrt{|\langle A^*Ax, x \rangle+\langle B^*Bx, x \rangle-(\langle Ax, x \rangle\langle x, Ax \rangle+\langle Bx, x \rangle\langle x, Bx \rangle)|}\\
			&=\sqrt{\left |\langle(A^*A+B^*B)x, x \rangle-\frac{\langle (A+B)x, x \rangle\langle x, (A+B)x \rangle+\langle (A-B)x, x \rangle\langle x, (A-B)x \rangle}{2}\right|}.
		\end{align*}
		\item Using initial calculations in (iii), 
		\begin{align*}
			\max\{\Delta_x(A), \Delta_x(B)\}&\geq \sqrt{\max\{|\langle Ax, Ax \rangle-\langle Ax, x \rangle\langle x, Ax \rangle|, |\langle Bx, Bx \rangle-\langle Bx, x \rangle\langle x, Bx \rangle|\}}\\
			&\geq \sqrt{|\langle Ax, Ax \rangle-\langle Ax, x \rangle\langle x, Ax \rangle-\langle Bx, Bx \rangle+\langle Bx, x \rangle\langle x, Bx \rangle|}\\
			&= \sqrt{|\langle A^*Ax, x \rangle-\langle B^*Bx, x \rangle-(\langle Ax, x \rangle\langle x,A x \rangle-\langle Bx, x \rangle\langle x, Bx \rangle)|}\\
			&= \sqrt{|\langle (A^*A-B^*B)x, x \rangle-(\langle Ax, x \rangle\langle x,A x \rangle-\langle Bx, x \rangle\langle x, Bx \rangle)|}
		\end{align*}
		\item Using the ultrametric inequality first and then using the non-Archimedean inner product we get
		\begin{align*}
			&\max\{\Delta_x(A), \Delta_x(B)\}=\max\{\|Ax- \langle Ax, x \rangle x \|, \|Bx- \langle Bx, x \rangle x \|\} \\
			&\geq \|Ax- \langle Ax, x \rangle x +Bx- \langle Bx, x \rangle x \| \\
			&\geq \sqrt{|\langle Ax- \langle Ax, x \rangle x +Bx- \langle Bx, x \rangle x, Ax- \langle Ax, x \rangle x +Bx- \langle Bx, x \rangle x\rangle |}\\
			&=\sqrt{|\langle Ax, Ax \rangle+\langle Bx,Bx \rangle+\langle Ax, Bx \rangle+\langle Bx, Ax \rangle-\langle Ax, x \rangle\langle x, Bx\rangle-\langle x, Ax \rangle\langle Bx, x\rangle-\langle Ax, x \rangle\langle x, Ax \rangle-\langle Bx, x \rangle\langle x, Bx \rangle|}\\
			&=\sqrt{|\langle (A+B)x, (A+B)x \rangle-\langle (A+B)x,x \rangle\langle x, (A+B)x \rangle|}.
		\end{align*}
	\item Using initial calculations in (v),
	\begin{align*}
		\max\{\Delta_x(A), \Delta_x(B)\}&=\max\{\|Ax- \langle Ax, x \rangle x \|, \|Bx- \langle Bx, x \rangle x \|\} \\
		&\geq \|Ax- \langle Ax, x \rangle x -Bx+ \langle Bx, x \rangle x \| \\
		&\geq \sqrt{|\langle Ax- \langle Ax, x \rangle x -Bx+ \langle Bx, x \rangle x, Ax- \langle Ax, x \rangle x -Bx+ \langle Bx, x \rangle x\rangle |}\\
		&=\sqrt{|\langle (A-B)x, (A-B)x \rangle-\langle (A-B)x,x \rangle\langle x, (A-B)x \rangle|}.	
	\end{align*}
	\end{enumerate}
\end{proof}
It is clear that the bounds obtained in Theorem \ref{PHRS} are sharp, since simple operators like the zero and identity operators yield equality in the inequalities.
\begin{remark}
	A fundamental difference between the Archimedean and non-Archimedean settings is that the ultrametric inequality replaces the usual triangle inequality. Consequently, product-type lower bounds analogous to the classical Heisenberg-Robertson and Schrodinger inequalities are generally unavailable. The estimates obtained in Theorem \ref{PHRS} should therefore be viewed as non-Archimedean counterparts preserving the operator-theoretic quantities appearing in the classical formulas (commutators and anti-commutators), while adapting their quantitative form to the geometry of non-Archimedean Hilbert spaces.	In this sense, the results identify the strongest bounds obtainable directly from the interaction between the non-Archimedean inner product and ultrametric norm structure.  Unlike the classical product bounds, the non-Archimedean case often yields weaker maximum bounds. This reflects the dominance of ultrametric norms and should be interpreted as a structural analogue rather than a direct equivalence. 
\end{remark}
We next derive non-Archimedean version of Theorem \ref{MP}.
\begin{theorem} (\textbf{Non-Archimedean Maccone-Pati Uncertainty Principles}) \label{PMP}
Let  $\mathcal{X}$ be a non-Archimedean Hilbert space. 	Let  $A:	\mathcal{D}(A)\subseteq \mathcal{X}\to  \mathcal{X}$ and $B:	\mathcal{D}(B)\subseteq \mathcal{X}\to  \mathcal{X}$  be possibly unbounded  linear operators. Then  for all $x \in \mathcal{D}(A)\cap  \mathcal{D}(B)$ with $\langle x, x \rangle=1$, we have 
	\begin{align*}
\Delta_x(A)+\Delta_x(B)\geq \max\{\Delta_x(A), \Delta_x(B)\} \geq      	|\langle (A+B)x, y \rangle |, \quad \forall y \in \mathcal{X} \text{ satisfying } \|y\|\leq 1, \langle x, y \rangle =0
	\end{align*}
and 	
	\begin{align*}
\Delta_x(A)+\Delta_x(B)\geq 	\max\{\Delta_x(A), \Delta_x(B)\} \geq      	|\langle (A-B)x, y \rangle |, \quad \forall y \in \mathcal{X} \text{ satisfying } \|y\|\leq 1, \langle x, y \rangle =0.
\end{align*}	
\end{theorem}
\begin{proof}
	Let $x \in \mathcal{D}(A)\cap  \mathcal{D}(B)$ be such that $\langle x,  x \rangle =1$. Let $y \in \mathcal{X}$ satisfies $\|y\|\leq 1 $ and  $\langle x, y \rangle =0.$ Then 
\begin{align*}
	|\langle (A\pm B)x, y \rangle |&=	|\langle (Ax-\langle Ax, x \rangle x)\pm (Bx-\langle Bx, x \rangle x), y \rangle |\\
	&\leq \|(Ax-\langle Ax, x \rangle x)\pm (Bx-\langle Bx, x \rangle x)\|\|y\|\\
	&\leq  \|(Ax-\langle Ax, x \rangle x)\pm(Bx-\langle Bx, x \rangle x)\|\\
	&\leq \max \{\|Ax-\langle Ax, x \rangle x\|, \|Bx-\langle Bx, x \rangle x\|\}\\
	&=\max\{\Delta_x(A), \Delta_x(B)\}.
\end{align*}
\end{proof}
\begin{remark}
	Theorem \ref{PMP} illustrates that, even in the non-Archimedean framework, one can formulate uncertainty relations that exploit orthogonality conditions to strengthen lower bounds. Unlike the classical Maccone–Pati principle, which relies on additive variance-type estimates, the non-Archimedean analogue replaces these with maximum-type bounds dictated by the ultrametric norm. This shift underscores the distinctive geometry of non-Archimedean Hilbert spaces: incompatibility between observables manifests not through multiplicative trade-offs but through dominance of the largest deviation. Consequently, the inequality should be interpreted as a structural analogue that preserves the operator-theoretic spirit of the classical relation while adapting its quantitative form to the non-Archimedean setting. In particular, the presence of auxiliary vectors orthogonal to the state vector highlights how orthogonality interacts with ultrametricity to yield sharper estimates than those available from direct commutator bounds alone.  
\end{remark}
\begin{example}
Let $\mathbb{K}$ be a non-Archimedean valued field.	Define 
\begin{align*}
	A\coloneqq	\begin{pmatrix}
		0&1\\
		0&0
	\end{pmatrix}\in \mathbb{M}_2(\mathbb{K}), \quad
	B\coloneqq	\begin{pmatrix}
		0&0\\
		1&0
	\end{pmatrix}\in \mathbb{M}_2(\mathbb{K}).
\end{align*}
Then $A$ and $B$ are non self-adjoint (with standard inner product). 	We find 
\begin{align*}
AB=\begin{pmatrix}
	1&0\\
	0&0
\end{pmatrix}, \quad 	BA=\begin{pmatrix}
0&0\\
0&1
\end{pmatrix}, \quad A+B=\begin{pmatrix}
0&1\\
1&0
\end{pmatrix}.
\end{align*}
Let $x=(a, b)\in \mathbb{K}^2$ be such that $\langle x, x \rangle=a^2+b^2=1$. We now find 
\begin{align*}
	Ax=\begin{pmatrix}
		b\\
		0
	\end{pmatrix}, \quad \langle Ax, x \rangle=ba
\end{align*}
and 
\begin{align*}
	Bx=\begin{pmatrix}
		0\\
		a
	\end{pmatrix}, \quad \langle Bx, x \rangle=ab.
\end{align*}
Hence
\begin{align*}
	\Delta _x(A)&= \|Ax- \langle Ax, x \rangle x \|=\max\{|b-a^2b|, |ab^2|\}\\
	&=|b|\max\{|1-a^2|, |ab|\}=|b|\max\{|b^2|, |ab|\}\\
	&=|b|^2\max\{|b|, |a|\}
\end{align*}
and 
\begin{align*}
	\Delta _x(B)&= \|Bx- \langle Bx, x \rangle x \|=\max\{|a^2b|, |a-ab^2|\}\\
	&=|a|\max\{|ab|, |1-b^2|\}=|a|\max\{|ab|, |a^2|\}\\
	&=|a|^2\max\{|a|, |b|\}.
\end{align*}
Therefore 
\begin{align*}
	\Delta _x(A)\Delta _x(B)&=|a^2b^2|\max\{|a^2|, |b^2|\}\geq |a^2b^2||a^2+b^2|\\
&=|a^2b^2|= |\langle Bx, Ax \rangle-\langle Ax, x \rangle \langle Bx, x \rangle|	
\end{align*}
and 
\begin{align*}
	\max\{\Delta_x(A), \Delta_x(B)\} &=\left(\max\{|a|, |b|\}\right)^3\\
	&\geq 	|\langle (A+B)x, y \rangle |, \quad \forall y \in \mathbb{K}^2 \text{ satisfying } \|y\|\leq 1, \langle x, y \rangle =0.
\end{align*}
\end{example}

\begin{example}
Let $\mathbb{K}$ be a non-Archimedean valued field.	Define 
	\begin{align*}
		A\coloneqq	\begin{pmatrix}
			1&0\\
			0&0
		\end{pmatrix}\in \mathbb{M}_2(\mathbb{K}), \quad
	B\coloneqq	\begin{pmatrix}
		0&0\\
		0&1
	\end{pmatrix}\in \mathbb{M}_2(\mathbb{K}).
	\end{align*}
Then $A$ and $B$ are self-adjoint. We find that 
\begin{align*}
	A^*=A=A^2, \quad B^*=B=B^2,\quad  AB=BA=0
\end{align*}
and 
\begin{align*}
	A+B=\begin{pmatrix}
		1&0\\
		0&1
	\end{pmatrix}.
\end{align*}
Let $x=(a, b)\in \mathbb{K}^2$ be such that $\langle x, x \rangle=a^2+b^2=1$. We now find 
\begin{align*}
	Ax=\begin{pmatrix}
		a\\
		0
	\end{pmatrix}, \quad \langle Ax, x \rangle=a^2
\end{align*}
and 
\begin{align*}
	Bx=\begin{pmatrix}
	0\\
	b
\end{pmatrix}, \quad \langle Bx, x \rangle=b^2.
\end{align*}
Hence 
\begin{align*}
		\Delta _x(A)&= \|Ax- \langle Ax, x \rangle x \|=\max\{|a-a^3|, |a^2b|\}\\
		&=|a|\max\{|1-a^2|, |ab|\}=|a|\max\{|b^2|, |ab|\}\\
		&=|ab|\max\{|b|, |a|\}
\end{align*}
and 
\begin{align*}
	\Delta _x(B)&= \|Bx- \langle Bx, x \rangle x \|=\max\{|b^2a|, |b-b^3|\}\\
	&=|b|\max\{|ab|, |1-b^2|\}=|b|\max\{|ab|, |a|^2\}\\
	&=|ab|\max\{|a|, |b|\}.
\end{align*}
Therefore 
\begin{align*}
		\Delta _x(A)\Delta _x(B)&=|a^2b^2|\max\{|a^2|, |b^2|\}\geq |a^2b^2||a^2+b^2|\\
		&=|a^2b^2|= |\langle ABx, x \rangle-\langle Ax, x \rangle \langle Bx, x \rangle|
\end{align*}
and 
\begin{align*}
	\max\{\Delta_x(A), \Delta_x(B)\} &=|ab|\max\{|a|, |b|\}\\
	&\geq 	|\langle (A+B)x, y \rangle |, \quad \forall y \in \mathbb{K}^2 \text{ satisfying } \|y\|\leq 1, \langle x, y \rangle =0.
\end{align*}
\end{example}
\begin{example}
Let $\mathbb{K}$ be a non-Archimedean valued field.	Define 
\begin{align*}
	A\coloneqq	\begin{pmatrix}
		1&1\\
		1&0
	\end{pmatrix}\in \mathbb{M}_2(\mathbb{K}), \quad
	B\coloneqq	\begin{pmatrix}
		0&1\\
		1&1
	\end{pmatrix}\in \mathbb{M}_2(\mathbb{K}).
\end{align*}
Then $A$ and $B$ are self-adjoint. 	We find 
\begin{align*}
	AB=\begin{pmatrix}
		1&2\\
		0&1
	\end{pmatrix}, \quad 	BA=\begin{pmatrix}
		1&0\\
		2&1
	\end{pmatrix}, \quad A+B=\begin{pmatrix}
		1&2\\
		2&1
	\end{pmatrix},\\
\{A,B\}=AB+BA=2\begin{pmatrix}
	1&1\\
	1&1
\end{pmatrix}, \quad 	[A,B]=AB-BA=2\begin{pmatrix}
	0&1\\
	-1&0
\end{pmatrix}.
\end{align*}
Let $x=(a, b)\in \mathbb{K}^2$ be such that $\langle x, x \rangle=a^2+b^2=1$. We now find 
\begin{align*}
	Ax=\begin{pmatrix}
		a+b\\
		a
	\end{pmatrix}, \quad \langle Ax, x \rangle=a^2+2ab
\end{align*}
and 
\begin{align*}
	Bx=\begin{pmatrix}
		b\\
		a+b
	\end{pmatrix}, \quad \langle Bx, x \rangle=b^2+2ab.
\end{align*}
If $\operatorname{Char}(\mathbb{K})\neq 2$, by making direct calculations, we see that inequalities in (ii) and (iii) of Theorem \ref{PHRS} hold.
\end{example}
\begin{remark}
Unlike in the classical Hilbert space case, the characterization of equality in the non-Archimedean Cauchy–Schwarz inequality is unknown. Therefore, the equality case in the non-Archimedean uncertainty principle cannot be characterized.
\end{remark}

\section{Conclusion and Open Problems}

For the first time in history, we have derived uncertainty principles for non-Archimedean Hilbert spaces. We observed surprising results which show the differences between the Archimedean and non-Archimedean worlds; namely, the sum in the Archimedean case is replaced by the maximum in the non-Archimedean case. Based on other famous uncertainty principles in Hilbert spaces, we raise the following problems.
\begin{enumerate}
	\item What is the non-Archimedean version of the Donoho–Stark–Elad–Bruckstein–Ricaud–Torresani support-size uncertainty principle \cite{ELADBRUCKSTEIN, RICAUDTORRESANI, DONOHOSTARK, ELADBOOK}?
	\item What is the non-Archimedean version of the Kraus-Maassen–Uffink entropic uncertainty principle \cite{KRAUS, MAASSENUFFINK}? In particular, what is the non-Archimedean version of the Deutsch entropic uncertainty principle \cite{DEUTSCH}?
	\item What is the non-Archimedean version of the Smith uncertainty principle and its generalizations \cite{SMITH, ALAGICRUSSELL, CHUANGNG, MESHULAM}?
	\item What is the non-Archimedean version of the Kuppinger-Durisi-Bölcskei support-size  uncertainty principle and its generalization \cite{KUPPINGER, STUDER}?
	\item What is the non-Archimedean version of the A. Wigderson-Y. Wigderson  uncertainty principle \cite{WIGDERSON}?
	\item What is the non-Archimedean version of the Goh-Goodman  uncertainty principle \cite{GOHGOODMAN}?
	\item What is the non-Archimedean version of the Jiang-Liu-Wu subfactor  uncertainty principle \cite{JIANGLIUWU}?
	\item What is the non-Archimedean version of the Bandeira-Lewis-Mixon	numerical sparsity  uncertainty principle \cite{BANDEIRALEWISMIXON}?
	\item  What is the non-Archimedean version of the uncertainty principle over finite fields \cite{EVRA, BORELLO, FENG}?
	\item What is the non-Archimedean version of the generalized  uncertainty principle \cite{BOSSOLUCIANO, TAWFIKDIAB, KEMPFMANGANOMANN}?
\end{enumerate}

\section{Acknowledgments}
This paper has been partially developed when the author attended the workshop “Quantum groups, tensor categories and quantum field theory”, held in University of Oslo, Norway from January 13 to 17, 2025. This event was organized by University of Oslo, Norway and funded by the Norwegian Research Council through the “Quantum Symmetry” project.\\
 Author  thanks the anonymous reviewer for his/her in depth  study of the manuscript and valuable suggestions which improved the manuscript.  

 \bibliographystyle{plain}
 \bibliography{reference.bib}

\end{document}